\documentclass[11pt,letterpaper]{amsart}

\usepackage{amssymb}
\usepackage{graphicx}
\usepackage{graphics}
\usepackage[margin=3.55cm]{geometry}
\usepackage{url}
\usepackage{amscd}
\usepackage{textcomp}
\usepackage{enumitem}
\usepackage[all]{xy}
\usepackage{mathtools} 
\usepackage{xcolor}
\usepackage{xspace}
\usepackage[normalem]{ulem} 
\usepackage[textsize=tiny]{todonotes}

\definecolor{darkblue}{rgb}{0,0,0.4} 
\usepackage{xr-hyper}
\usepackage[colorlinks=true, citecolor=darkblue, filecolor=darkblue, linkcolor=darkblue,urlcolor=darkblue]{hyperref} 
\usepackage[all]{hypcap}

\usepackage{bbm} 

\usepackage{dsfont}        

\usepackage{caption}
\usepackage{comment}

\DeclarePairedDelimiter\ceil{\lceil}{\rceil}

\newtheorem{theorem}{Theorem}[section]
\newtheorem{lemma}[theorem]{Lemma}

\newtheorem{corollary}[theorem]{Corollary}
\theoremstyle{remark}
\newtheorem{remark}[theorem]{Remark}
\theoremstyle{definition}
\newtheorem{definition}[theorem]{Definition}
\numberwithin{equation}{section}

\renewcommand{\theta}{\vartheta}

\newcommand{\set}[1]{\left\{#1\right\}}
\newcommand{\ball}{\mathds{B}}
\newcommand{\sphere}{\mathds{S}}
\newcommand{\CP}[1]{\mathds{CP}^{#1}}
\newcommand{\bCP}[1]{\overline{\mathds{CP}^{#1}}}
\newcommand{\del}{\partial}
\newcommand{\nbd}{\mathcal{N}}
\newcommand{\m}{m} 

\newcommand{\bicol}{\varphi}
\newcommand{\LBT}{(T,\left\{\Pi_v\right\})}
\newcommand{\DP}{\mathcal{D}}

\newcommand{\T}{\mathds{T}}

\DeclareMathOperator{\Int}{Int}
\newcommand{\Slice}{\mathcal{S}}

\begin{document}

\title{Unknotting number 21 knots are slice in $K3$}

\author{Marco Marengon}%
\address{Alfr\'ed R\'enyi Institute of Mathematics, Budapest, Hungary}%
\email{\href{mailto:marengon@renyi.hu}{marengon@renyi.hu}}%

\author{Stefan Mihajlovi\'c}
\address{Central European University, Budapest, Hungary}%
\email{\href{mailto:mihajlovic.stefan@renyi.hu}{mihajlovic.stefan@renyi.hu}}

\begin{abstract}
    We prove that all knots with unknotting number at most 21 are smoothly slice in the $K3$ surface. We also prove a more general statement for 4-manifolds that contain a plumbing tree of spheres.
    Our strategy is based on a flexible method to remove double points of immersed surfaces in 4-manifolds by tubing over neighbourhoods of embedded trees.
    As a byproduct, we recover a classical result of Norman and Suzuki that every knot is smoothly slice in $\sphere^2 \times \sphere^2$ and in $\CP2 \# \bCP2$.
\end{abstract}

\maketitle

\section{Introduction}

Given a smooth 4-manifold $X$, a knot $K \subset \sphere^3$ is called smoothly \emph{slice in $X$} if it bounds a properly embedded smooth disc $D \subset X^\circ := X-\Int \ball^4$. An analogous definition can be given in the topological category, but unless otherwise stated, we assume that a 4-manifold $X$ is smoooth and that \emph{slice} (without further modifiers) means \emph{smoothly slice}.
One focus of this paper is on a smooth closed 4-manifold known as the $K3$ surface, for which we prove the following.


\begin{theorem}
\label{thm:K3-intro}
    Every knot $K$ in $\sphere^3$ with unknotting number $u(K) \leq 21$ is smoothly slice in $K3$.
\end{theorem}

Let $\Slice(X)$ denote the set of knots that are (smoothly) slice in $X$.
Depending on the 4-manifold, $\Slice(X)$ can coincide with the set of classically slice knots $\Slice(\sphere^4)$ (e.g.\ for $X=\sphere^4$, $\sphere^1 \times \sphere^3$, or $\T^4$), with the set of all knots (e.g.\ for $\sphere^2 \times \sphere^2$ or $\CP2 \# \bCP2$, cf.\ \cite{N:slice, S:slice}), or can differ from both of them.

So far $\CP2$ (together with $\bCP2$) is the only known example of a simply connected closed 4-manifold with $\Slice(\sphere^4) \subsetneq \Slice(X) \subsetneq \set{\mathrm{knots}}$, by a result of Yasuhara \cite{Y:CP2}.
A key property of $\CP2$ used by Yasuhara is that only finitely many homology classes in $H_2(\CP2)$ are primitive (not multiples of other classes); this is no longer true for e.g.\ $\CP2\#\CP2$, or more generally $n\CP2$, and this is the main reason why we have no other examples of 4-manifolds with $\Slice(\sphere^4) \subsetneq \Slice(X) \subsetneq \set{\mathrm{knots}}$.
In particular, the case of the simplest simply connected closed 4-manifold which is not homeomorphic to a connected sum of $\CP2$, $\bCP2$, and $\sphere^2 \times \sphere^2$ - namely the $K3$ surface, is also wide open.

The question whether $\Slice(K3) \neq \set{\mathrm{knots}}$ was raised in \cite[Question 6.1]{MMP:slice}. Prior to our work, it was shown in \cite[Corollary 2.8]{MMP:slice} that every knot with unknotting number $\leq 2$ is slice in $K3$. This result was later strengthened in unpublished works of Mukherjee and Stipsicz-Szab\'o to unknotting number $\leq 4$ and $\leq 9$, respectively.
Our new bound in Theorem \ref{thm:K3-intro} follows from the following more general theorem and from the existence of a plumbing tree of 22 spheres smoothly embedded in $K3$.

\begin{theorem}
\label{thm:plumbing}
        If there is a plumbing tree of $n$ smooth (resp.\ locally flat) spheres $S = S_1 \cup \ldots \cup S_n$ embedded into a smooth (resp.\ topological) 4-manifold $X^4$, then any knot $K  \subset \sphere^3$ with 4-dimensional clasp number $c_4(K) \leq n-1$ (resp.\ $c_4^{\mathrm{top}}(K) \leq n-1$) is smoothly (resp.\ topologically) slice in $X^4$.
\end{theorem}


The smooth (resp.\ topological) 4-dimensional clasp number $c_4(K)$ (resp.\ $c_4^{\mathrm{top}}(K)$) appearing in the statement above is the minimum number of self-intersections of a smooth (resp.\ locally flat) normally immersed disc in $\ball^4$ with boundary $K$. The inequalities $c_4^{\mathrm{top}}(K) \leq c_4(K) \leq u(K)$ hold for every knot $K$. For more details see Section \ref{sec:normal}.

The main technique of our paper is a more general \textit{tubing} procedure for immersed surfaces in 4-manifolds.
The usual tubing procedure, appearing e.g.\ in \cite{N:slice}, consists of removing pieces of our immersed surfaces contained in two balls centred at two double points,
and connecting the resulting Hopf links (obtained by intersecting the surfaces with the boundaries of said balls) with two tubes.
In comparison to the standard case, we use more general links that are mirrors of each other, which are obtained by intersecting the surfaces with the boundary of a neighborhood of an embedded 1-complex. See Figure \ref{fig:sketchIntro}.

The $K3$ surface is a very natural example to consider, and a historically important one. On one hand it has a simple Morse theoretical description, with no 1- and 3-handles. On the other hand it has a rich geometric structure (it is symplectic, and in fact a Kähler surface) and already displays all the exotic complications of dimension 4 having infinitely many exotic copies. Because of this, understanding $\Slice(K3)$ could hint at the more general behaviour of 4-manifolds with non-trivial Seiberg-Witten invariants. Finally, the importance of $K3$ can be seen noting that if the 11/8-conjecture were true \cite{11/8}, then any simply connected smooth 4-manifold would be homeomorphic to a connected sum of some number of $\CP2$, $\bCP2$, $\sphere^2 \times \sphere^2$ and $K3$.

Furthermore, understanding $\Slice(K3)$ can shed light on whether sliceness detects exotic pairs, i.e.\ if there exists homeomorphic 4-manifolds whose smooth types are distinguished by their sets of slice knots. This question is \cite[Question 6.2]{MMP:slice}, and is motivated by the hope that sliceness could be used to disprove the 4-dimensional Poincar\'e conjecture \cite{FGMW:mnm}.

While the question is still open, it was recently shown that \emph{H-sliceness} (another generalisation of the notion of classical sliceness to all 4-manifolds) does indeed detect the exotic pair given by $K3 \# \bCP2$ and $3\CP2 \# 20 \bCP2$ \cite[Corollary 1.5]{MMP:slice}. Note that every knot is slice in $3\CP2 \# 20 \bCP2$. Then, an example of a knot that is not slice in $K3$ would be a great step towards showing that the above exotic pair is detected by sliceness too. On the other hand, if every knot is slice in $K3$, then there is no hope to distinguish this exotic pair by sliceness.

We mentioned earlier that our proof is based on the existence of a plumbing tree of 22 spheres smoothly embedded in the $K3$ surface \cite[Proposition 1]{FM:K3}.
Such a plumbing tree was used by Finashin-Mikhalkin to build a $(-86)$-framed sphere in $K3$ \cite[Theorem 1]{FM:K3}.
Their result was recently expanded by Stipsicz-Szab\'o \cite[Theorem 1.1]{SSz:spheres}, who exhibited plumbing trees of spheres in all elliptic surfaces $E(n)$, and used them to produce very negative spheres in $E(n)$. $E(2)=K3$ recovers the result of Finashin-Mikhalkin, and using the plumbing trees from \cite{SSz:spheres} we can in fact prove the following result for $E(n)$:

\begin{corollary}
    \label{cor:En-intro}
    For $n\geq 2$, every knot $K$ with 4-dimensional clasp number
    \[
    c_4(K) \leq 11 \cdot n - \ceil*{\frac n5}
    \]
    is smoothly slice in $E(n)$.
\end{corollary}


For $n=1$, $E(1) = \CP2 \#9\bCP2$, so every knot is slice in $E(1)$. For $n=2$ we recover Theorem \ref{thm:K3-intro} (with $u(K) \leq 21$ replaced by the more general $c_4(K) \leq 21$).


Our technique to prove Theorem \ref{thm:plumbing} (which we outline in Section \ref{sec:outline} below) is quite flexible and amenable to different applications. Recall that, given a knot $K \subset \sphere^3$ and a smooth 4-manifold $X$, the \emph{slice genus in $X$} of $K$ is defined as
\[
g_X(K) = \min\set{g(\Sigma) \,|\, \Sigma \xhookrightarrow{\mathrm{sm}} X^\circ, \Sigma \textrm{ compact}, \del \Sigma = K }.
\]

A classical argument of Norman (outlined in \cite[Lemma 1]{N:slice}) implies that for a smooth, closed 4-manifold $X$ with a 0-framed sphere the function $g_X$ is bounded: just take an immersed surface $\Sigma$ with $\partial \Sigma = K$ sitting in a homology class dual to the 0-framed sphere $S_0$, and use parallel copies of $S_0$ to resolve all double points by tubing.

We re-interpret this result using our technique.
\begin{theorem}[\cite{N:slice}]
\label{thm:0-sphere-intro}   
    If a smooth 4-manifold $X^4$ contains a smoothly embedded 0-framed sphere $S$ with a geometrically dual smooth closed surface $S^*$, then for every knot $K$
    \[g_X(K)\leq g(S^*).\]
\end{theorem}

We remark that if $X$ is closed, and the homology class $[S]$ is primitive, a geometrically dual surface always exists.

Note that both $\sphere^2 \times \sphere^2$ and $\mathds{CP}^{2} \# \overline{\mathds{CP}^{2}}$ have a 0-framed sphere with a geometrically dual sphere. Then, as a corollary of Theorem \ref{thm:0-sphere-intro} we recover the following classical results of Norman and Suzuki, which we already mentioned above.
\begin{corollary}[{\cite{N:slice, S:slice}}]
    \label{cor:NS}
    Every knot is smoothly slice in $\sphere^2 \times \sphere^2$ and $\CP2 \# \bCP2$.
\end{corollary}

\begin{remark}
\label{rem:topological}
    The careful reader will have noted that we stated Theorems \ref{thm:K3-intro} and \ref{thm:0-sphere-intro} and Corollary \ref{cor:En-intro} only in the smooth case.
    This is because by Freedman's classification theorem every closed, indefinite, simply connected 4-manifold contains a topological $\sphere^2 \times \sphere^2$ or $\CP2 \# \bCP2$ summand, so every knot is topologically slice in it.
\end{remark}

\subsection{Organisation}
In Section \ref{sec:outline} we give a short outline of the proof of Theorem \ref{thm:plumbing}, which is the main result of this paper.
Section \ref{sec:technical} is the technical part of the paper: we introduce locally bipartitioned trees and explain how they can be used to remove double points of normally immersed surfaces. In Section \ref{sec:applications} we prove Theorems \ref{thm:plumbing} and \ref{thm:0-sphere-intro} and their corollaries.

\subsection{Conventions}
%
%
Given a link $L$ in $\sphere^3$, $\m(L)$ denotes its mirror image, and if $L$ is oriented, $-L$ denotes the mirror image with orientation reversed on each component.

\subsection{Acknowledgements}
We warmly thank Andr\'as Stipsicz for his encouragement and suggestions, and Danica Kosanovi\'c, Lisa Piccirillo, and Arunima Ray for their comments on a draft of this paper.
We also thank the anonymous referees for their suggestions.

MM acknowledges that:
This project has received funding from the European Union’s Horizon 2020 research and innovation programme under the Marie Sk{\l}odowska-Curie grant agreement No.\ 893282.

SM would like to thank the Alfréd Rényi Institute of Mathematics for its hospitality and acknowledges that:
The  research  supporting  this article was partly sponsored  by Central European University Foundation of Budapest (CEUBPF). The theses explained herein  represent  the  ideas  of  the  author, and do not necessarily reflect the views of CEUBPF.

\section{Outline of the proof of Theorem \ref{thm:plumbing}}
\label{sec:outline}
We remark that an independent proof of Theorem \ref{thm:K3-intro}, which we deduce from Theorem \ref{thm:plumbing} in this paper, will appear in future work of the second named author.

\begin{figure}
    \includegraphics[width=0.98\textwidth]{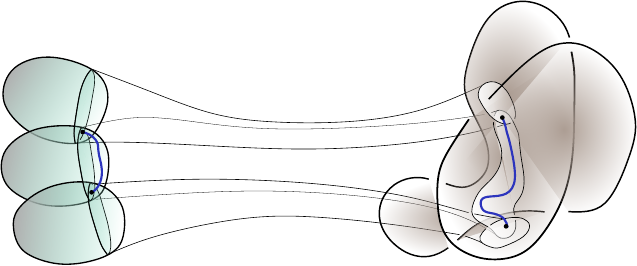}
    \centering
    \caption{Sketch of the argument: for simplicity, the embedded tree consists of an edge connecting 2 vertices.
    We remove neighbourhoods of this tree from both the sphere plumbing (left) and the immersed disk (right), and get 3-chain links as new boundary components.
    We connect these via 3 tubes to get rid of all self-intersections. }
    \label{fig:sketchIntro}
\end{figure}

The key ingredient of our proof is the existence of a plumbing tree $S = S_1 \cup \ldots \cup S_n$ of $n$ spheres smoothly embedded in $X^\circ$.
Given a knot with $c_4 \leq n-1$, there is an immersed disc $D$ with $n-1$ double points in $\ball^4$, and hence in a collar neighbourhood of $\del(X^\circ) = \sphere^3$. The bulk of our technical work lies in finding two disjoint 4-balls $\ball_1$ and $\ball_2$ in $X^\circ$, one containing all the double points of $D$ and the other all the double points of $S$, such that the links $D \cap \del\ball_1$ and $S \cap \del\ball_2$ are mirrors of each other. Once that is done, we remove $D \cap \ball_1$ and $S \cap \ball_2$ from the surfaces and tube what is left of them to obtain a smoothly embedded surface (with no double points).

In our argument, the balls $\ball_1$ and $\ball_2$ are chosen as regular neighbourhoods of the same tree embedded in $D$ and $S$ respectively, so that the vertices of the embedded trees are exactly the double points. We will need to endow the trees with \emph{local bipartitions} (cf.\ Definition \ref{def:btree}) to keep track of the two local sheets of the surface near the double points (cf.\ Definition \ref{def:embedding}).

We show that given a plumbing tree of spheres there is always a suitable locally bipartitioned tree embedded therein (cf.\ Lemma \ref{lem:treeinplumbing}), and that the same plumbing tree can be re-embedded also in the immersed disc $D$ (cf.\ Lemma \ref{lem:treeinconnectedsurface}). To conclude, we note that a locally bipartitioned tree embedded in a normally immersed surface $\Sigma$ completely determines the link $\Sigma \cap \del\ball$ (cf.\ Lemma \ref{lem:linkofembedding}), and therefore we can tube and eliminate all the double points.

\section{Immersed surfaces and locally bipartitioned trees}
\label{sec:technical}

\subsection{Normal immersions}
\label{sec:normal}
Let $X$ be a smooth 4-manifold (possibly with boundary), and let $\widetilde{\Sigma}$ be a compact surface (possibly with boundary).
A smooth immersion $i \colon \widetilde{\Sigma}^2 \to X^4$ is called \emph{normal} if $i(\widetilde{\Sigma}) \cap \del X = i(\del \widetilde{\Sigma})$, $i$ is transverse to $\del X$, and all self-intersections of $i(\widetilde{\Sigma})$ are transverse double points in $\Int X$.
In such a case, we call $\Sigma := i(\widetilde{\Sigma})$ a \emph{normally immersed surface}.
We denote the set of double points by $\DP(\Sigma) \subset \Sigma$.

Following Shibuya \cite{S:clasp}, we define the \emph{4-dimensional clasp number} of a knot $K \subset \sphere^3$ as
\[
c_4(K) :=
\min\set{
|\DP(\Delta)|\,|\,\Delta \subset \ball^4 \mathrm{\,\,normally\,\, immersed\,\, disc}, \del \Delta = K
}.
\]
There is an analogous \emph{topological} (i.e., as opposed to \emph{smooth}) counterpart, defined via locally flat normal immersions, and denoted $c_4^{\mathrm{top}}$ in \cite{FP:clasp}, but we focus on the smooth version $c_4$ (see Remark \ref{rem:topological}).

There are inequalities
\[
g_4(K) \leq c_4(K) \leq u(K),
\]
where $g_4$ is the (smooth) slice genus (in $\ball^4$), and $u$ is the unknotting number. Any non-trivial slice knot $K$ gives an example where $c_4(K) = 0$ and $u(K)\neq 0$. As for the other inequality, there are examples $K_n$ with $g_4(K_n)=n$ and $c_4(K_n) \geq 2n$ \cite{JZ:clasp, FP:clasp}, but it is unknown whether $c_4 \leq 2 g_4$ for all knots.
For relations of $c_4$ with the slicing number and the concordance unknotting number we refer the reader to \cite{OS:clasp}.

\subsection{Tubing self-intersections}
A standard technique in 4-dimensional topology is to remove double points of normally immersed surfaces by tubing. Given two immersed surfaces, and a double point of each of them, one can remove parts of the surfaces contained in small 4-balls centred at the two double points: since the surfaces intersect the boundary 3-spheres in Hopf links, these can be tubed (connected by cylinders) to create a new surface with two fewer double points.

The main idea of this paper is to perform a tubing operation such as the one appearing in \cite{N:slice} over more complex links.

\begin{figure}
    \includegraphics[width=0.9\textwidth]{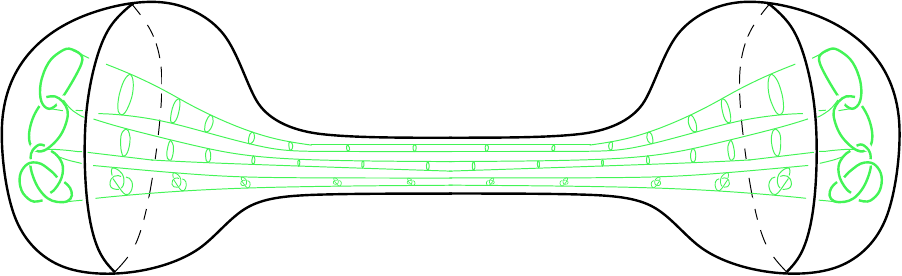}
    \centering
    \caption{Sketch of the tubing operation.}
    \label{fig:multitube}
\end{figure}

\begin{lemma}
\label{lem:tubing}
Let $\Sigma_1, \Sigma_2 \subset X$ be two smooth normally immersed surfaces into a smooth connected 4-manifold $X$. Let $\ball_1$ and $\ball_2$ be two disjoint 4-balls with boundaries $\sphere_1$ and $\sphere_2$, respectively, such that $\ball_1 \cap \Sigma_2 = \ball_2 \cap \Sigma_1 = \varnothing$. If the links $L_1 := \Sigma_1 \cap \sphere_1$ and $L_2 := \Sigma_2 \cap \sphere_2$ are mirrors of each other, we can eliminate all the self-intersections of $\Sigma_1 \cup \Sigma_2$ in $\ball_1 \cup \ball_2$ and build a new normally immersed surface $\Sigma \subset X$ by \emph{tubing} - connecting $\Sigma_1$ and $\Sigma_2$ via disjoint annuli. If the initial surfaces are oriented, and $L_1 = -L_2$ with the induced orientations, then $\Sigma$ inherits a natural orientation too.
\end{lemma}

\begin{proof}
    Pick an arc $\gamma \subset X \setminus (\ball_1 \cup \ball_2 \cup \Sigma_1 \cup \Sigma_2)$ connecting $\sphere_1$ and $\sphere_2$.
    If $\nbd(\gamma)$ denotes a small regular neighbourhood of $\gamma$, then $\ball_1 \cup \ball_2 \cup \nbd(\gamma)$ is a 4-ball $\ball_3$, and the link $(\Sigma_1 \cup \Sigma_2) \cap \del\ball_3$ is $L_1 \sqcup L_2$.
    
    If $L_1 = \m(L_2)$, then $L_1 \sqcup L_2$ bounds $|L_1|$ disjoint annuli in $\ball_3$, which can be glued to $(\Sigma_1 \cup \Sigma_2) \setminus (\ball_1 \cup \ball_2)$ to obtain $\Sigma'$.
    See Figure \ref{fig:multitube} for an illustration.

    If $\Sigma_1$ and $\Sigma_2$ are oriented and $L_1 = - L_2$, then the annuli can be oriented coherently and the result of the gluing is oriented too. 
\end{proof}

We will apply Lemma \ref{lem:tubing} to pairs of 4-balls arising as regular neighbourhoods of trees embedded in a normally immersed surface. The next subsection is devoted to defining and explaining the objects we consider.

\subsection{Locally bipartitioned trees}

Given a graph $\Gamma$, we denote the set of its vertices by $V(\Gamma)$ and the set of its edges by $E(\Gamma)$. Given a vertex $v \in V(\Gamma)$, we denote the set of edges adjacent to $v$ by $E(v)$.

\begin{definition}
\label{def:btree}
A \emph{locally bipartitioned tree} $\left(T, \left\{\Pi_v\right\}_{v \in V(T)}\right)$ is given by
\begin{itemize}
    \item a finite tree $T$ and,
    \item for each vertex $v \in V(T)$, a set $\Pi_v = \set{A_v, B_v}$ which gives a bipartition of $E(v)$ (i.e., $E(v) = A_v \sqcup B_v$).
\end{itemize}

Given $v \in V(T)$ and $e \in E(v)$, we let $\pi_v(e) \in \Pi_v$ denote the element of the bipartition containing the edge $e$.
\end{definition}

\begin{remark}
A tree $T$ together with a bicolouring $\bicol \colon E(T) \to \set{0,1}$ of its edges naturally induces a locally bipartitioned tree, by considering for each $v \in V(T)$ the partition of $E(v)$ defined by the colour of the edge.

Vice versa, given a locally bipartitioned tree $\LBT$, it is always possible to find a bicolouring of $T$ that induces the local bipartitions $\Pi_v$. (This follows by induction, using the fact that each non-empty tree has a leaf.)

In all statements and proofs we follow the locally bipartitioned perspective, but in the figures for simplicity we always represent locally bipartitioned trees by bicolourings, and we use red and blue for the two colours. 
\end{remark}

Given a locally bipartitioned tree $\LBT$, in Definition \ref{def:associatedlink}, we will associate a link $L\LBT $ in $\sphere^3$ with it.
Its geometric meaning is the following: given a \emph{suitable embedding} of $\LBT $ into an immersed surface $\Sigma \subset X^4$ (as defined in Definition \ref{def:embedding}), $L\LBT $ is the link defined by intersecting $\Sigma$ with the boundary of a regular neighbourhood of $T$ (see Lemma \ref{lem:linkofembedding}).

\begin{definition}
\label{def:associatedlink}
Given a locally bipartitioned tree $\LBT $, its \emph{associated link} is the unoriented link $L\LBT $ in $\sphere^3$ defined in two steps as follows:
\begin{enumerate}
    \item for each vertex $v \in V(T)$, take a Hopf link with the two components labelled by the two elements of $\Pi_v$;
    \item for each edge $e \in E(T)$, connecting vertices $v$ and $w$, connect sum the two Hopf links associated with $v$ and $w$ at the components labelled $\pi_v(e)$ and $\pi_w(e)$.
\end{enumerate}
\end{definition}
See Figure \ref{fig:associatedlink} for an illustration of a locally bipartitioned tree and its associated link.

\begin{figure}
    \includegraphics[width=0.9\textwidth]{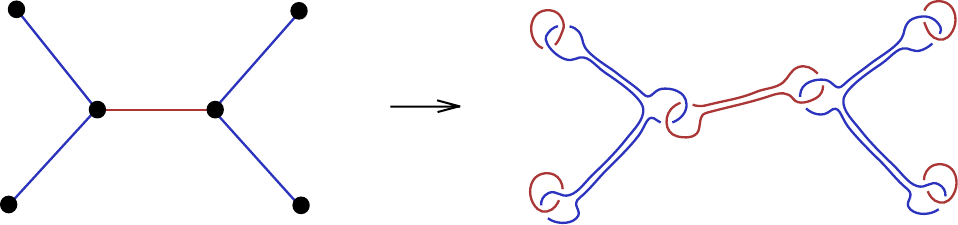}
    \centering
    \caption{An example of a locally bipartitioned tree and its associated link in $\sphere^{3}$.}
    \label{fig:associatedlink}
\end{figure}

\begin{remark}
\label{rem:mirror}
    Let $\LBT$ be any locally bipartitioned tree. As an (unoriented) link, $L\LBT $ is isotopic to its mirror image. This can be easily checked for the Hopf link, and for the general case it follows from the commutativity of the connected sum operation.
\end{remark}

\begin{definition}
\label{def:embedding}
    Let $\LBT $ be a locally bipartitioned tree and $\Sigma^2 \subset X^4$ be a normally immersed surface, and let $\DP(\Sigma)$ denote the set of double points of $\Sigma$.
    A \emph{suitable embedding} of $\LBT $ into $\Sigma$ is an embedding
    \[
    f \colon T \to \Sigma
    \]
    such that
    \begin{itemize}
    \item $f^{-1}(\DP(\Sigma)) = V(T)$, and
    \item for each vertex $v \in V(T)$, any two edges in $E(v)$ from the same element of the bipartition $\Pi_v$ map into the same local component of $\Sigma$, whereas any two edges from the two different elements of the bipartition $\Pi_v$ map into two different local components of $\Sigma$.
    \end{itemize}
\end{definition}
A plumbing tree of surfaces is one example of an immersed surface, see Figure \ref{fig:embeddedGraph} for an illustration of a suitable embedding in this case.

%
%

\begin{figure}
    \includegraphics[width=\textwidth]{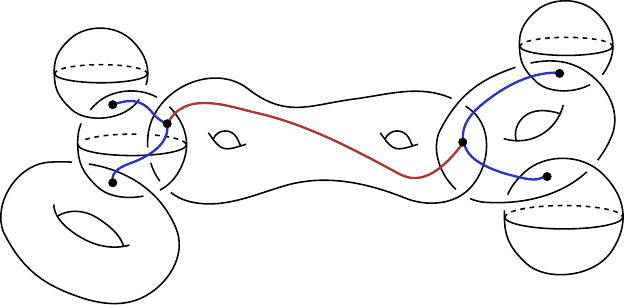}
    \centering
    \caption{Sketch of a surface plumbing and a suitable embedding of the locally bipartitioned tree from Figure \ref{fig:associatedlink}.}
    \label{fig:embeddedGraph}
\end{figure}


\begin{lemma}
\label{lem:linkofembedding}
    Let $\Sigma^2 \subset X^4$ be a normally immersed surface into a smooth 4-manifold, and let $f$ be a suitable embedding of a locally bipartitioned tree $\LBT $ into $\Sigma$.
    Then the link of the embedding, i.e., the intersection of $\Sigma$ with $\del\nbd(f(T)) \cong \sphere^3$ is exactly the associated link $L\LBT $.
\end{lemma}
%
%

\begin{proof}
For simpler notation, in the proof of this lemma we will identify the tree $T$ with the image of the embedding $f$, so embedded vertices $f(v)$ will be denoted by just $v$, and embedded arcs $f(e)$ will be denoted by $e$.

A small neighbourhood of a tree is a 4-ball $\mathds{B}_{0}$, and we only need to identify the link we get by intersecting $\Sigma$ with the boundary 3-sphere $\del\ball_0$.
%
We visualise the whole 4-ball neighbourhood $\ball_0$ of the tree $T$ as the union of 4-balls $\ball_v$ centred at each vertex $v$, and 4-dimensional solid tubes joining two 4-balls $\ball_{v_i}$ and $\ball_{v_j}$, one for each edge.
Each such solid tube is a standard neighbourhood of an edge $e$ from $v_i$ to $v_j$ minus the balls $\ball_{v_i}$ and $\ball_{v_j}$.
%

Near the vertex $v$ the surface $\Sigma$ locally looks like two planes intersecting transversely, hence $\Sigma \cap \del \ball_v$ is a Hopf link, with its two components lying on the two different local sheets of $\Sigma$.

Now consider two vertices $v_1$ and $v_2$ that are connected by an edge $e \subset \Sigma$. The first bullet point in Definition \ref{def:embedding} ensures that $e$ does not intersect any double points except $v_1$ and $v_2$.
The intersection of $\Sigma$ with the boundary of a small 4-ball neighbourhood $\nbd(e)$ of $e$ in $X \setminus (\ball_{v_1} \cup \ball_{v_2})$ consists of two arcs parallel to the edge, and two smaller arcs, one in each $\del\ball_{v_i}$, shared with the Hopf component at $v_i$ with label $\pi_{v_i}(e)$.
Thus, removing $\nbd(e)$ results into connect summing the two Hopf links along the components labelled $\pi_{v_1}(e)$ and $\pi_{v_2}(e)$.
\end{proof}

\subsection{From surface plumbings to locally bipartitioned trees}
In the next lemma we show how to find a large locally bipartitioned tree suitably embedded into a plumbing tree of surfaces. See Figure \ref{fig:embeddedGraph} for an illustration.


\begin{lemma}
\label{lem:treeinplumbing}
    Let $\Sigma = \Sigma_1 \cup \ldots \cup \Sigma_n$ be a plumbing tree of $n$ closed surfaces in a smooth 4-manifold $X^4$.
    Then there exists a locally bipartitioned tree $\LBT$ with exactly $n-1$ vertices and a suitable embedding $f \colon T \to \Sigma$.
    
    Moreover, given any such $\LBT$ and $f$:
    \begin{itemize}
    \item $f(T)$ contains all the $n-1$ double points of $\Sigma$;
    \item for each $i=1, \ldots, n$, $\Sigma_i \cap f(T)$ is contractible.
    \end{itemize}
\end{lemma}

\begin{proof}
    Leaving the local bipartitions aside for a moment, the fact that there is a tree $T$ with $n-1$ vertices and an embedding $f$ such that $f^{-1}(\DP(\Sigma)) = V(T)$ follows by a simple induction argument on $n$, using the fact that any plumbing tree has a leaf.
    
    Then, once we have such a tree and such an embedding, we choose as local bipartitions $\Pi_v$ exactly the ones induced by the local sheets of $\Sigma$, so that the embedding $f$ is automatically suitable.
    
    Now suppose that we are given a locally bipartitioned tree $\LBT$ with exactly $n-1$ vertices and a suitable embedding $f \colon T \to \Sigma$. Then the first bullet point follows from the fact that $T$ has exactly $n-1$ vertices, $\Sigma$ has exactly $n-1$ double points, and $f$ is one-to-one.
    Finally, the second bullet point follows from the fact that $\Sigma$ is a plumbing \emph{tree}, so if two double points on $\Sigma_i$ are connected by a path in $\Sigma \cap f(T)$, that path intersects $\Sigma_i$ in a connected subset of $f(T)$, which is itself contractible.
\end{proof}

%
%
%

\subsection{Locally bipartitioned trees in immersed connected surfaces}

With a variation of the previous argument, we can show the following result.

\begin{lemma}
\label{lem:treeinconnectedsurface}
    Let $X$ be a smooth 4-manifold (possibly with boundary), let $i \colon \widetilde\Sigma \to X^4$ be a normal immersion of a connected surface with $m$ self-intersections, and let $\Sigma = i(\widetilde\Sigma)$. Then any locally bipartitioned tree $\LBT $ with $\ell \leq m$ vertices can be suitably embedded in $\Sigma$.
\end{lemma}

\begin{proof}
    The proof is by induction on $\ell$. The base cases $\ell=1$ and $\ell=2$ are straightforward.
    %
    %
    
    For $\ell > 2$ start by choosing a leaf $v$ of $T$, and let $e$ be the edge connecting it to some vertex $w$. We define $T'$ by setting $V(T') = V(T) - \set{v}$, and $E(T') = E(T) - \set{e}$.
    Then $(T', \left\{\Pi_v\right\})$ satisfies the induction hypothesis so we can find a suitable embedding $f' \colon T' \to \Sigma$; note that the number of self-intersections of $\Sigma$ in the image of $f'$ is strictly less than $m$.
    We now wish to define $f \colon T \to \Sigma$ by setting $f|_{T'} = f'$ and so that the remaining vertex $v$ maps to a self-intersection of $\Sigma$ not in the image of $f'$. The only subtle point is how to define $f(e)$ so that the bipartition of $E(w)$ coming from the local sheets of $\Sigma$ is exactly $\Pi_w$.

    If we let $i \colon \widetilde\Sigma \to X^4$ be the normal immersion, then for any vertex $z\in V(T')$ we can identify its two preimages $\set{\widetilde{z_1},\widetilde{z}_2} = i^{-1}(f'(z)) \subset \widetilde{\Sigma}$ with the partition $\Pi_z = \set{A_z, B_z}$, so that the lift $i^{-1}(f'(e'))$ of an edge $e' \in E(z)$ has $\pi_z(e')$ as one of its endpoints.
    (The careful reader will note that for this point we need $\ell>2$.)
    See Figure \ref{fig:preimage} for an illustration.

    \begin{figure}
        \centering
        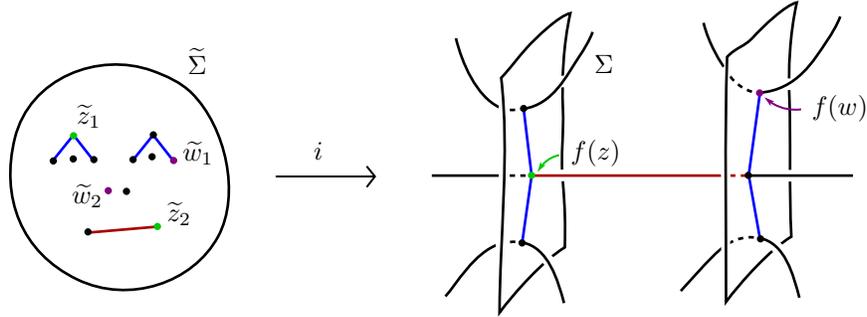
        \caption{The figure shows a normal immersion $i$ of a connected surface $\widetilde{\Sigma}$. The image $\Sigma$ is represented by the black lines and planes on the right. The tree $T$ from Figure \ref{fig:associatedlink} is suitably embedded in $\Sigma$. On the left we drew the pre-image of $T$ through $i$.
        The pre-image of a vertex $f(z)$ is a two-point set $\set{\widetilde{z_1},\widetilde{z}_2} \subset \widetilde{\Sigma}$, which we identify with $\Pi_z = \set{A_z, B_z}$.
        }
        \label{fig:preimage}
    \end{figure}
    
    From here one can see that the preimage of the embedded tree $i^{-1}(f'(T'))$ consists of a disjoint union of contractible components embedded in $\Int \widetilde\Sigma$. (In fact, upon choosing a compatible bicolouring of $T'$, every maximal monochromatic subtree of $T'$ will give rise to a connected tree embedded in $\widetilde\Sigma$. See Figure \ref{fig:preimage}.)
    
    Choose an arc $\gamma$ connecting the preimage $\pi_w(e) \in i^{-1}(f'(w))$ to any of the preimages of $v$. By the previous paragraph, we can choose it so that its interior avoids $i^{-1}(f'(T'))$, and of course also the preimages of the double points. Thus, we can define an embedding $f$ by setting $f(e) = i(\gamma)$. The choice of the starting point of $\gamma$ (namely $\pi_w(e)$) was made so that the bipartition of $E(w)$ coming from the local sheets of $\Sigma$ is exactly $\Pi_w$, and therefore $f$ defines a suitable embedding of $\LBT$ into $\Sigma$.
\end{proof}

\section{Applications}
\label{sec:applications}

{
\renewcommand{\thetheorem}{\ref{thm:plumbing}}
\begin{theorem}
        If there is a plumbing tree of $n$ smooth (resp.\ locally flat) spheres $S = S_1 \cup \ldots \cup S_n$ embedded into a smooth (resp.\ topological) 4-manifold $X^4$, then any knot $K  \subset \sphere^3$ with 4-dimensional clasp number $c_4(K) \leq n-1$ (resp.\ $c_4^{\mathrm{top}}(K) \leq n-1$) is smoothly (resp.\ topologically) slice in $X^4$.
\end{theorem}
\addtocounter{theorem}{-1}
}

\begin{proof}
    The knot $K$ bounds a normally immersed disk $D$ with exactly $n-1$ self-intersections in $\sphere^{3} \times [0,1]$ (if necessary, add extra self-intersections), and thus also in a collar neighbourhood of $\del(X \setminus \Int \ball^4)$.
    We can assume that the disk $D \subset X \setminus \Int\ball^4$ is disjoint from the plumbing tree $S$.

    We use Lemma \ref{lem:treeinplumbing} to find a suitable embedding $f_S$ of a locally bipartitioned tree $\LBT $ with $n-1$ vertices into $S$. By Lemma \ref{lem:treeinconnectedsurface}, the same tree can be suitably embedded into the disc $D$.
    
    By Lemma \ref{lem:linkofembedding}, the links we get on the boundaries of small neighbourhoods of these two embeddings of $\LBT $ are the same. Finally, as these links are isotopic to their mirror images, we can apply Lemma \ref{lem:tubing} and tube the self-intersections of $D$ with the plumbing intersections, thus removing all of them.
    Note that the second bullet point of Lemma \ref{lem:treeinplumbing} implies that after removing a small ball about $f_S(T)$ what is left of each sphere $S_i$ is a disc; therefore, we do not add any genus to the immersed disk $D$ with the tubing procedure. This way, after tubing, we get a properly embedded disk $D' \subset X \setminus \Int \ball^4$ with boundary $K$.
\end{proof}   


Theorem \ref{thm:plumbing} immediately implies that every knot is slice in $S^2 \times S^2$, since $S^2 \times S^2$ contains arbitrarily long plumbing trees of spheres, obtained by taking a copy of the first $S^2$ factor and several disjoint copies of the second $S^2$ factor.
(A similar reasoning works for $\CP2 \# \overline{\CP2}$.)
Later we will establish a more general result, see Theorem \ref{thm:0-sphere-intro}.

We now restate and prove Corollary \ref{cor:En-intro} from the introduction. Theorem \ref{thm:K3-intro} from the introduction follows immediately from it by recalling that $E(2) = K3$ and $c_4 \leq u$.

{
\renewcommand{\thetheorem}{\ref{cor:En-intro}}
\begin{corollary}
    For $n\geq 2$, every knot $K$ with 4-dimensional clasp number
    \[
    c_4(K) \leq 11 \cdot n - \ceil*{\frac n5}
    \]
    is (smoothly) slice in $E(n)$.
\end{corollary}
\addtocounter{theorem}{-1}
}

\begin{proof}
    By \cite[Proof of Theorem 1.1]{SSz:spheres}, $E(n)$ contains a plumbing tree of $11\cdot n - \ceil*{\frac n5} + 1$ smooth spheres.
    Thus, we can apply Theorem \ref{thm:plumbing} to conclude.

    In the special case of $E(2) = K3$, the existence of a plumbing tree of 22 spheres was in fact proved earlier in \cite{FM:K3}. See Figure \ref{fig:22spheres} for an illustration of such a plumbing tree.
    %
    %
\end{proof}
        
     \begin{figure}
             \centering
             \includegraphics[width=0.7\textwidth]{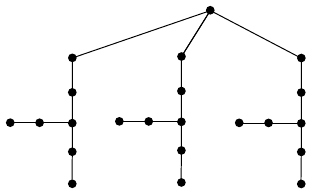}
             \caption{A plumbing tree of 22 spheres in $E(2) = K3$. This is the example in \cite[Figure 3]{FM:K3}, which can be realised by viewing $K3$ as a singular elliptic fibration and by taking three $\widetilde{E}_{6}$ fibers and a section.}
             \label{fig:22spheres}
    \end{figure}

The previous result was an application of Theorem \ref{thm:plumbing} to a family of symplectic 4-manifolds, which always have non-vanishing Seiberg-Witten invariants.
By contrast, the next result is for 4-manifolds with a 0-framed sphere, whose Seiberg-Witten invariants are always vanishing. This result follows from \cite[Lemma 1]{N:slice}, but we give here a proof using the techniques that we just developed.

Recall that every closed, connected, embedded surface $\Sigma$ in a closed 4-manifold $X$ admits a geometrically dual surface $\Sigma^*$ if its homology class is primitive. To find a dual surface, first take an algebraically dual immersed surface (which exists by Poincar\'e duality), and resolve its intersections to make it embedded. As the algebraic intersection between this embedded dual surface and $\Sigma$ is +1, pair the positive and negative intersections except for one. Finally, add tubes for each pair to cancel intersections with $\Sigma$ until there is only one left.

{\renewcommand{\thetheorem}{\ref{thm:0-sphere-intro}}
\begin{theorem}[{\cite{N:slice}}]
    If a smooth 4-manifold $X$ contains a smoothly embedded 0-framed sphere $S$ with a geometrically dual surface $S^*$, then for every knot $K$
    \[g_X(K)\leq g(S^*).\]
\end{theorem}
\addtocounter{theorem}{-1}
}
\begin{proof}
    The proof is very similar to the proof of Theorem \ref{thm:plumbing}, but in this case we exhibit an arbitrarily large plumbing of surfaces, one of which does not have to be a sphere.
    
    Take any knot $K \subset \sphere^3 = \del(X^4 \setminus \Int\ball^4)$, with unknotting number $n \in \mathds{N}$.
    We use the fact that the sphere $S$ is 0-framed so that we are able to take parallel push-offs $S_{i}$ where $S = S_{0}$. Define a plumbing tree of $n+1$ surfaces by taking $n$ parallel copies $S_{i}$, and one copy of the dual surface $S^*$, see Figure \ref{fig:pushoffsAndDual}. By Lemma \ref{lem:treeinplumbing}, we can find a suitable embedding of a locally bipartitioned tree $\LBT$ with $n$ vertices into the plumbing. 
    (In this case there is a simple one, namely a uniformly coloured linear tree $T$ with $n$ vertices and $n-1$ edges obtained by choosing an arc in the dual surface $S^*$ that intersects all the $n$ parallel copies of $S$.)
    
    By Lemma \ref{lem:treeinconnectedsurface}, we can suitably re-embed $\LBT $ into an immersed disk $D$ with $n$ self-intersections contained in a collar neighbourhood of $\del(X \setminus \Int\ball^4)$.

    Finally, using Lemma \ref{lem:tubing}, we tube the self-intersections of $D$ with the plumbing intersections, removing all of them. The only difference with the proof of Theorem \ref{thm:plumbing} is that in this case the tubing operation does add genus, because we tube with the surface $S^*$.
    Since each surface in the plumbing tree is attached via a single tube to $D$, we can orient each of them so that the result will be oriented. Thus, we have constructed a properly embedded surface in $X\setminus \Int\ball^4$ with genus $g(S^*)$ and boundary $K$.
\end{proof}

    \begin{figure}
            \centering
            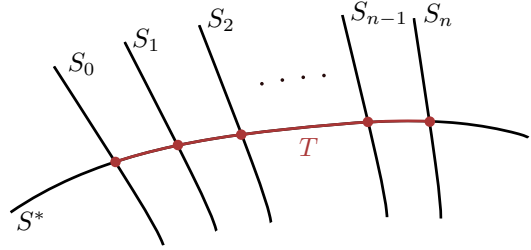
            \caption{The figure shows the push-offs $S_{i}$ of and the geometrically dual surface $S^{*}$.}
            \label{fig:pushoffsAndDual}
    \end{figure}

%
%
    %
    %
    %

\bibliographystyle{alpha}
\bibliography{bibliography}

\end{document}